\documentclass[12pt]{amsart}
\usepackage{amssymb, amsthm, tikz}
\usepackage{float}
\usepackage{mathtools}
\usepackage{hyperref}
\usepackage{fullpage}
\usepackage{cleveref}
\usepackage{tabularx}
\usepackage{comment}

\newtheorem*{theorem*}{Theorem}
\newtheorem{theorem}{\textbf{Theorem}}[section]
\newtheorem{proposition}[theorem]{\textbf{Proposition}}
\newtheorem{corollary}[theorem]{\textbf{Corollary}}
\newtheorem{lemma}[theorem]{\textbf{Lemma}}

\theoremstyle{definition}
\newtheorem{definition}[theorem]{Definition}

\newtheorem{example}[theorem]{Example}

\numberwithin{equation}{section}

\DeclareMathOperator{\hh}{H}
\DeclareMathOperator{\vv}{V}

\DeclareMathOperator{\FS}{\mathfrak{S}}

\tikzstyle{path}=[line width=5pt, black, opacity=0.75]

\newcommand{\grid}[2]{

    \foreach \i in {1,...,#1}{
        \pgfmathtruncatemacro\row{int(#1-\i+1)}
        \draw (0,2*\i-1) node[xshift=-0.3cm] {$\row$} -- (2*#2, 2*\i-1);
    }

    \foreach \i in {1,...,#2}{
        \pgfmathtruncatemacro\col{\i)}
        \draw (2*\i-1,0) -- (2*\i-1,2*#1) node[yshift=0.3cm] {$\col$};
    }
}

\title{Switch Operators for the Six-Vertex Model}
\author{Evelyn Choi, Jadon Geathers, Slava Naprienko}

\begin{document}

\begin{abstract}
    In this paper, we introduce and analyze a new \textit{switch operator} for the six-vertex model. This operator, derived from the Yang-Baxter equation, allows us to express the partition function with arbitrary boundaries in terms of a base case with domain wall boundary conditions. As an application, we derive explicit formulas for the factorial Schur functions and their generalizations. Our results provide new insights into the relationship between boundary conditions and partition functions in the six-vertex model.
\end{abstract}

\maketitle
\section{Introduction}
The six-vertex model is a widely studied statistical mechanics model due to its connections to various areas of mathematics, including representation theory, combinatorics, and integrable systems. It was first introduced by Pauling in \cite{Pau35}. 

In this paper, we view the six-vertex model as a combinatorial system of paths on a lattice model under prescribed boundary conditions. Each combination of paths forms an admissible state of the model, in which we assign vertex weights to the vertices in each state. The partition function, which is a weighted sum over all possible configurations, encodes the statistical properties of the model.

A recurrent problem in the combinatorics of integrable lattice models is to find appropriate Boltzmann weights that lead to meaningful and useful values in the partition function. If these weights satisfy the Yang-Baxter equation (see \cite{Bax82}), then the partition function satisfies functional equations and can be both computed and identified with special functions from literature. For multiple instances, see \cite{BBF11,BMN14,ABPW21,Mot15,Mot17,Mot17ik,N23} and references therein.

One of the key challenges in studying the six-vertex model is understanding the dependence of the partition function on the boundary conditions, which are the states assigned to the vertices on the boundary of the lattice. Functional relations for the partition functions with different boundary conditions exist by the Yang-Baxter equation. In this paper, we introduce a novel method for analyzing this dependence: the \textit{switch operators}. These operators, derived from the Yang-Baxter equation, allow us to express the partition function with arbitrary boundary conditions in terms of a base case with simple domain wall boundary conditions. 

The main result is the following: 
\begin{theorem*}
    The partition function $Z_{\alpha,\beta}$ with right boundary $\alpha$ and top boundary $\beta$ is expressed in terms of the base case $Z_{\delta,\delta}$ as follows:
    \[
        Z_{\alpha,\beta} = \partial_\alpha^{\hh} \partial_\beta^{\vv}\left(Z_{\delta,\delta}\right),
    \]
    where $\partial_\alpha^{\hh}$ and $\partial_\beta^{\vv}$ are the switch operators.
\end{theorem*}

As an application, we compute the explicit expression for the factorial Schur functions and their generalizations. We also prove that these functions are asymptotically symmetric in column parameters. These results extend and generalize those from Section 7 of \cite{BMN14}.

\bigskip
   
\textbf{Acknowledgements.} This paper was created through the 2022 Stanford Undergraduate Research in Mathematics (SURIM) program. We would like to thank everyone involved in organizing SURIM, and in particular we thank Lernik Asserian for directing the program.

\section{Six-vertex model}
In this section, we review the six-vertex model from statistical mechanics and introduce notation. For a treatment from the point of view of statistical mechanics, see \cite{Bax82}. Here, we represent the model as a rectangular lattice with paths traveling from northwest to southeast. That is, paths enter the model from the top or the left and leave the model on the right and bottom. Paths may intersect, but their movement is restricted to the right and downward directions. Thus, there are only six admissible states for a vertex, illustrated in \Cref{fig:sixtypies}.

\begin{figure}
    \centering
    \begin{tabularx}{0.8\textwidth} { 
      | >{\centering\arraybackslash}X 
      | >{\centering\arraybackslash}X 
       | >{\centering\arraybackslash}X 
         | >{\centering\arraybackslash}X 
           | >{\centering\arraybackslash}X 
      | >{\centering\arraybackslash}X | }
     \hline
    \begin{tikzpicture}[baseline={([yshift=-.3ex]current bounding box.center)}]
    \draw[black](-1,0) -- (1,0);
    \draw[black] (0, 1) -- (0,-1);
    \end{tikzpicture} &
    \begin{tikzpicture}[baseline={([yshift=-.3ex]current bounding box.center)}]
    \draw[black](-1,0) -- (1,0);
    \draw[black] (0, 1) -- (0,-1);
    \draw[red, line width = 2.3](-1,0) -- (1,0);
    \draw[red, line width = 2.3] (0, 1) -- (0,-1);
    \end{tikzpicture} 
    & 
    \begin{tikzpicture}[baseline={([yshift=-.3ex]current bounding box.center)}]
    \draw[black](-1,0) -- (1,0);
    \draw[black] (0, 1) -- (0,-1);
    \draw[red, line width = 2.3] (0,1) -- (0, -1);
    \end{tikzpicture}
    & 
    \begin{tikzpicture}[baseline={([yshift=-.3ex]current bounding box.center)}]
    \draw[black](-1,0) -- (1,0);
    \draw[black] (0, 1) -- (0,-1);
    \draw[red, line width = 2.3] (1,0) -- (-1, 0);
    \end{tikzpicture}& 
    \begin{tikzpicture}[baseline={([yshift=-.3ex]current bounding box.center)}]
    \draw[black](-1,0) -- (1,0);
    \draw[black] (0, 1) -- (0,-1);
    \draw[red, line width = 2.3] (1,0) -- (0, 0) -- (0, 1);
    \end{tikzpicture}& 
    \begin{tikzpicture}[baseline={([yshift=-.3ex]current bounding box.center)}]
    \draw[black](-1,0) -- (1,0);
    \draw[black] (0, 1) -- (0,-1);
    \draw[red, line width = 2.3] (-1,0) -- (0, 0) -- (0, -1);
    \end{tikzpicture}\\
     \hline
     \;$a_1$\; & $a_2$  & $b_1$ & $b_2$ & $c_1$ & $c_2$ \\
    \hline
    \end{tabularx}
    \caption{Six admissible states named $a_1,a_2,b_1,b_2,c_1,c_2$ following \cite{Bax82}}
    \label{fig:sixtypies}
\end{figure}
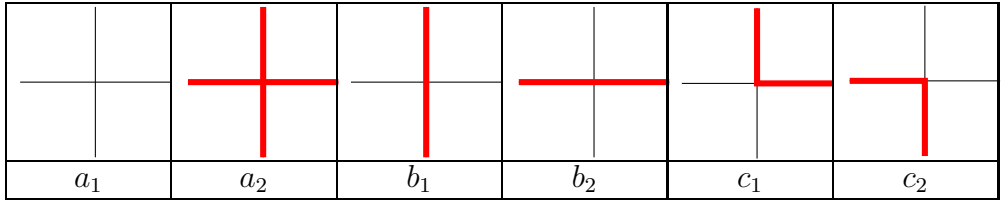

Defining a six-vertex model requires specifying the size of the grid, as well as the boundary conditions (at what positions the paths enter or leave the grid). Given such a model, we have a system of admissible configurations of paths. Each configuration is called a \textit{state}, and each state consists only of the six mentioned vertices. The weight of a given type of vertex at a certain position in the lattice is represented using the following weight functions:
\[
    a_1,a_2,b_1,b_2,c_1,c_2\colon \mathbb{Z} \times \mathbb{Z} \to \mathbb{Z}.
\]

For each state $s$, its weight $w(s)$ is given by computing the product of weights over all vertices in the state. Then the \textit{partition function} $Z(\mathfrak{S})$ of a six-vertex model $\mathfrak{S}$ is the sum of weights of all admissible states:
\[
    Z(\mathfrak{S}) = \sum_{\text{states}}\prod_{\text{vertices}}W\text{(vertex)}.
\]

More generally, we consider the six-vertex model with row labels $I = (I_1,I_2,\dots,I_n)$, column labels $J = (J_1,J_2,\dots,J_m)$, and boundary conditions $\beta = (\beta^l,\beta^t,\beta^r,\beta^b)$ for the left, top, right, and bottom boundaries, respectively. We denote such a model by 
\[
    \FS(I;J;\beta) = \FS(I_1,\dots,I_n; J_1,\dots,J_m; \beta^l, \beta^t, \beta^r, \beta^b).
\]

For brevity, we denote the partition function of such a model by
\[
    Z(I;J; \beta^l,\beta^t,\beta^r,\beta^b) = Z(\mathfrak{S}(I;J; \beta^l,\beta^t,\beta^r,\beta^b)).
\]

Lastly, we use the notation $[n] = (1, 2, \dots, n)$. 
\begin{definition}\label{def:dwbc}[DWBC]
    The six vertex model with \textit{domain wall boundary conditions} (abbreviated as DWBC) is defined as $\FS_n^{\operatorname{DWBC}} = \FS([n];[n];\emptyset,[n],[n],\emptyset)$. This is an $n \times n$ lattice with the following boundaries: paths enter from the top edge of each column and exit from the right edge of each row. The partition function for this model is denoted as $Z_n^{\operatorname{DWBC}} = Z(\FS_n^{\operatorname{DWBC}})$. An illustration of the DWBC model is shown in Figure \ref{fig:dwbc}. 
\end{definition}

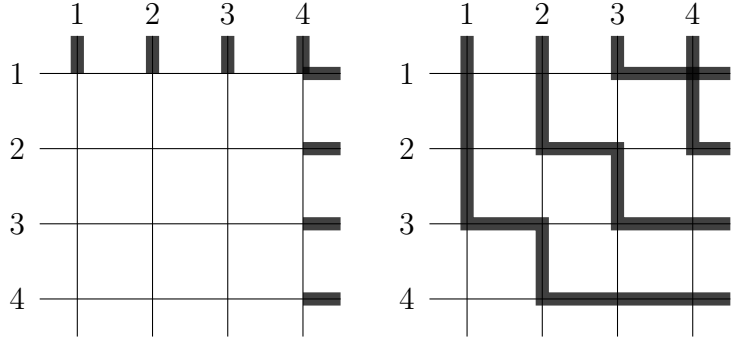
\begin{figure}
    \centering
    \begin{tikzpicture}[xscale=0.5, yscale=0.5]
        \grid{4}{4}
        
        \draw[path] (1,8) to (1,7);
        \draw[path] (3,8) to (3,7);
        \draw[path] (5,8) to (5,7);
        \draw[path] (7,8) to (7,7);
        
        \draw[path] (7,1) to (8,1);
        \draw[path] (7,3) to (8,3);
        \draw[path] (7,5) to (8,5);
        \draw[path] (7,7) to (8,7);
    \end{tikzpicture} \quad 
    \begin{tikzpicture}[xscale=0.5, yscale=0.5]
        \grid{4}{4}
        
        \draw[path] (1,8) to (1,3) to (3,3) to (3,1) to (8,1);
        \draw[path] (3,8) to (3,5) to (5,5) to (5,3) to (8,3);
        \draw[path] (5,8) to (5,7) to (8,7);
        \draw[path] (7,8) to (7,5) to (8,5);
    \end{tikzpicture}
    \caption{The model $\mathfrak{S}_n^{\operatorname{DWBC}}$ for $n=4$ and a typical state in the model.}
    \label{fig:dwbc}
\end{figure}

\begin{figure}
    \centering
    \begin{tabular}[t]{c|c}
    state & weight \\
    
    $\vcenter{\hbox{\begin{tikzpicture}[scale=0.4]
        \grid{3}{3}
        
        \draw[path] (1,6) to (1,5) to (3,5) to (3,3) to (5,3) to (5,1) to (6,1);
        \draw[path] (3,6) to (3,5) to (5,5) to (5,3) to (6,3);
        \draw[path] (5,6) to (5,5) to (6,5);
    \end{tikzpicture}}}$ &
    $\begin{aligned}[c]
        c_1(1,1)a_2(1,2)a_2(1,3)\\
        a_1(2,1)c_1(2,2)a_2(2,3)\\
        a_1(3,1)a_1(3,2)c_1(3,3)
    \end{aligned}$
    \\
    $\vcenter{\hbox{\begin{tikzpicture}[scale=0.4]
        \grid{3}{3}
        
        \draw[path] (1,6) to (1,5) to (3,5) to (3,3) to (3,1) to (6,1);
        \draw[path] (3,6) to (3,5) to (5,5) to (5,3) to (6,3);
        \draw[path] (5,6) to (5,5) to (6,5);
    \end{tikzpicture}}}$ &
    $\begin{aligned}[c]
        c_1(1,1)a_2(1,2)a_2(1,3)\\
        a_1(2,1)b_1(2,2)c_1(2,3)\\
        a_1(3,1)c_1(3,2)b_2(3,3)
    \end{aligned}$\\
    
    $\vcenter{\hbox{\begin{tikzpicture}[scale=0.4]
\grid{3}{3}

        \draw[path] (1,6) to (1,5) to (1,3) to (3,3) to (3,1) to (6,1);
        \draw[path] (3,6) to (3,5) to (5,5) to (5,3) to (6,3);
        \draw[path] (5,6) to (5,5) to (6,5);
    \end{tikzpicture}}}$ &
    $\begin{aligned}[c]
        b_1(1,1)c_1(1,2)a_2(1,3)\\
        c_1(2,1)c_2(2,2)c_1(2,3)\\
        a_1(3,1)c_1(3,2)b_2(3,3)
    \end{aligned}$
    \\
    $\vcenter{\hbox{\begin{tikzpicture}[scale=0.4]
        \grid{3}{3}
        
        \draw[path] (1,6) to (1,1) to (6,1);
        \draw[path] (3,6) to (3,3) to (6,3);
        \draw[path] (5,6) to (5,5) to (6,5);
    \end{tikzpicture}}}$ &
    $\begin{aligned}[c]
        b_1(1,1)b_1(1,2)c_1(1,3)\\
        b_1(2,1)c_1(2,2)b_2(2,3)\\
        c_1(3,1)b_2(3,2)b_2(3,3)
    \end{aligned}$ \\
    $\vcenter{\hbox{\begin{tikzpicture}[scale=0.4]
    \grid{3}{3}
        \draw[path] (1,6) to (1,3) to (6,3);
        \draw[path] (3,6) to (3,1) to (6,1);
        \draw[path] (5,6) to (5,5) to (6,5);
    \end{tikzpicture}}}$ &
    $\begin{aligned}[c]
        b_1(1,1)b_1(1,2)c_1(1,3)\\
        c_1(2,1)a_2(2,2)b_2(2,3)\\
        a_1(3,1)c_1(3,2)b_2(3,3)
    \end{aligned}$
    \\
    $\vcenter{\hbox{\begin{tikzpicture}[scale=0.4]
    \grid{3}{3}
        \draw[path] (1,6) to (1,3) to (6,3);
        \draw[path] (3,6) to (3,5) to (6,5);
        \draw[path] (5,6) to (5,1) to (6,1);
    \end{tikzpicture}}}$ &
    $\begin{aligned}[c]
        b_1(1,1)c_1(1,2)a_2(1,3)\\
        c_1(2,1)b_2(2,2)a_2(2,3)\\
        a_1(3,1)a_1(3,2)c_1(3,3)
    \end{aligned}$
    \\
    $\vcenter{\hbox{\begin{tikzpicture}[scale=0.4]
    \grid{3}{3}
        \draw[path] (1,6) to (1,1) to (6,1);
        \draw[path] (3,6) to (3,5) to (6,5);
        \draw[path] (5,6) to (5,3) to (6,3);
    \end{tikzpicture}}}$ &
    $\begin{aligned}[c]
        b_1(1,1)c_1(1,2)a_2(1,3)\\
        b_1(2,1)a_1(2,2)c_1(2,3)\\
        c_1(3,1)b_2(3,2)b_2(3,3)
    \end{aligned}$
    \\
    
    \end{tabular}
    \caption{Admissible states with Boltzmann weights}
    \label{fig:states}
\end{figure}

\begin{example}
    Let $n = 3$. Then there are seven admissible configurations in $\mathfrak{S}_n^{\operatorname{DWBC}}$. See \Cref{fig:states} for the complete list of admissible states with their Boltzmann weights. The partition function $Z_3^{\operatorname{DWBC}}$ then is the sum of all the weights of the configurations. Note that we write the product of weights aligned with the positions where they occur in the model for convenience.
    
    With the seven admissible configurations of the $3 \times 3$ lattice and all the defined vertex weights, we can then compute the partition function by summing all seven products.
\end{example}

\begin{definition}[Yang-Baxter]
We say that a six-vertex model is \textit{integrable} if, given its weight functions $a_1,\dots,c_2$, there exist new weight functions 
\begin{align*}
    a_1^{\hh},a_2^{\hh},b_1^{\hh},b_2^{\hh},c_1^{\hh},c_2^{\hh}\colon \mathbb{Z} \times \mathbb{Z} \to \mathbb{C},\\
    a_1^{\vv},a_2^{\vv},b_1^{\vv},b_2^{\vv},c_1^{\vv},c_2^{\vv}\colon \mathbb{Z} \times \mathbb{Z} \to \mathbb{C},
\end{align*}
such that the \textit{Yang-Baxter equation} holds, i.e. we have the equality of the following partitions:

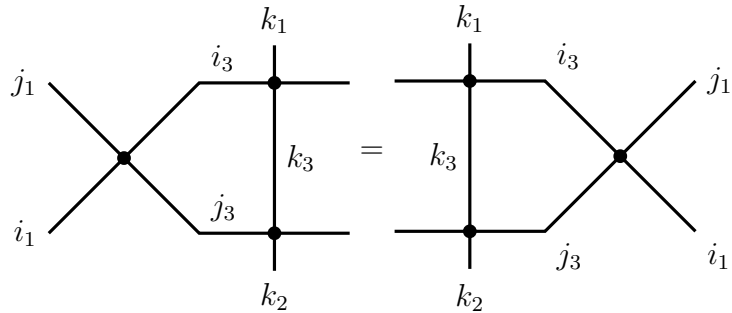
\tikzset{vertex/.style={draw, circle, fill=black, minimum size=4pt, inner sep=0pt}}

\tikzset{vertex/.style={draw, circle, fill=black, minimum size=4pt, inner sep=0pt}}

\begin{equation*}
    \begin{tikzpicture}[anchor=base, baseline=(current bounding box.center), line width=1.2pt]
    
    \node[vertex] at (1,1) {};
    \node[vertex] at (3,0) {};
    \node[vertex] at (3,2) {};
    
    \node[left] at (0,0) {$i_1$};
    \node[left] at (0,2) {$j_1$};
    \node[above right] at (2,2) {$i_3$};
    \node[above right] at (2,0) {$j_3$};
    \node[above] at (3, 2.5) {$k_1$};
    \node[below] at (3, -0.5) {$k_2$};
    \node[right] at (3, 1) {$k_3$};
    
    \draw (0,0) -- (2,2) -- (4,2);
    \draw (0,2) -- (2,0) -- (4,0);
    \draw (3,-0.5) -- (3,1) -- (3,2.5);
    
    \end{tikzpicture} = 
    \begin{tikzpicture}[anchor=base, baseline=(current bounding box.center), line width=1.2pt]
    
    \node[vertex] at (3,1) {};
    \node[vertex] at (1,2) {};
    \node[vertex] at (1,0) {};
    
    \node[below right] at (4,0) {$i_1$};
    \node[right] at (4,2) {$j_1$};
    \node[above right] at (2,2) {$i_3$};
    \node[below right] at (2,0) {$j_3$};
    \node[above] at (1, 2.5) {$k_1$};
    \node at (1, -1) {$k_2$};
    \node[left] at (1, 1) {$k_3$};
    
    \draw (0,0) -- (2,0) -- (4,2);
    \draw (0,2) -- (2,2) -- (4,0);
    \draw (1, -0.5) -- (1, 2.5);
    
    \end{tikzpicture}
\end{equation*}

In the graphical equation, $i_1,j_1,k_1,i_2,j_2,k_2 \in \{0,1\}$ are fixed, and $i_3,j_3,k_3 \in \{0,1\}$ must iterate through all indices. An index of $1$ indicates the existence of a path, while $0$ indicates its absence.

Similarly, we define vertically integrable weights using the graphical equation:

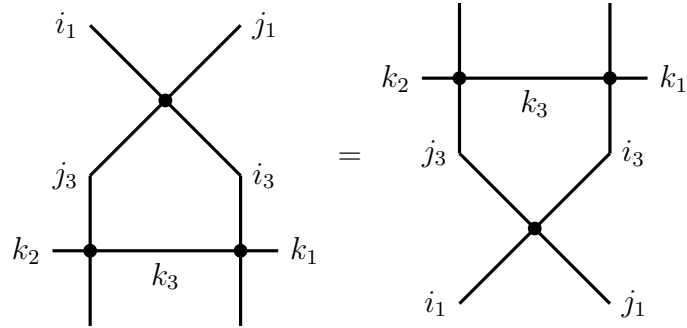
\begin{equation*}
    \begin{tikzpicture}[anchor=base, baseline=(current bounding box.center), line width=1.2pt]
    
    \node[vertex] at (1, 2) {};
    \node[vertex] at (0,0) {};
    \node[vertex] at (2,0) {};
    
    \node[left] at (0,3) {$i_1$};
    \node[right] at (2,3) {$j_1$};
    \node[right] at (2, 1) {$i_3$};
    \node[left] at (0, 1) {$j_3$};
    \node[right] at (2.5, 0) {$k_1$};
    \node[left] at (-0.5, 0) {$k_2$};
    \node[below] at (1, 0) {$k_3$};
    
    \draw (0, 3) -- (1,2) -- (2, 3);
    \draw (0, 1) -- (1,2) -- (2,1);
    \draw (0,1) -- (0,0);
    \draw (2, 1) -- (2, 0);
    \draw (-0.5, 0) -- (2.5, 0);
    \draw (0,0) -- (0, -1);
    \draw (2, 0) -- (2, -1);

    \end{tikzpicture} = 
    \begin{tikzpicture}[anchor=base, baseline=(current bounding box.center), line width=1.2pt]
    
    \node[vertex] at (5, 0) {};
    \node[vertex] at (4,2) {};
    \node[vertex] at (6,2) {};
    
    \node[left] at (4,-1) {$i_1$};
    \node[right] at (6,-1) {$j_1$};
    \node[right] at (6, 1) {$i_3$};
    \node[left] at (4, 1) {$j_3$};
    \node[right] at (6.5, 2) {$k_1$};
    \node[left] at (3.5, 2) {$k_2$};
    \node[below] at (5, 2) {$k_3$};
    
    \draw (4, -1) -- (5,0) -- (6, -1);
    \draw (4, 1) -- (5,0) -- (6,1);
    \draw (4,1) -- (4,2);
    \draw (6, 1) -- (6, 2);
    \draw (3.5, 2) -- (6.5, 2);
    \draw (4,2) -- (4, 3);
    \draw (6, 2) -- (6, 3);
    
    \end{tikzpicture}
\end{equation*}

where $i_1, j_1, k_1, i_2, j_2, k_2 \in \{ 0,1\}$ are again fixed, and $i_3, j_3, k_3 \in \{0, 1\}$ iterate through all their possibles values.  

\end{definition}

The introduction of cross vertices to facilitate the Yang-Baxter equation, as illustrated in the above figure, gives rise to horizontal and vertical cross vertices. We provide the six allowable horizontal cross vertices of the six-vertex model in \Cref{fig:hcross}. Similarly, the six allowable vertical cross vertices of the six-vertex model are shown in \Cref{fig:vcross}.

\begin{figure}
    \centering
    \begin{tabularx}{0.8\textwidth} { 
      | >{\centering\arraybackslash}X 
      | >{\centering\arraybackslash}X 
       | >{\centering\arraybackslash}X 
         | >{\centering\arraybackslash}X 
           | >{\centering\arraybackslash}X 
      | >{\centering\arraybackslash}X | }
     \hline
     
    \begin{tikzpicture}[baseline={([yshift=-.3ex]current bounding box.center)}]
    \draw[black](-1,1) -- (1,-1);
    \draw[black] (-1, -1) -- (1, 1);
    \end{tikzpicture} &
    \begin{tikzpicture}[baseline={([yshift=-.3ex]current bounding box.center)}]
    \draw[black](-1,1) -- (1,-1);
    \draw[black] (-1, -1) -- (1, 1);
    \draw[red, line width = 2.3](-1,1) -- (1,-1);
    \draw[red, line width = 2.3] (-1, -1) -- (1, 1);
    \end{tikzpicture} 
    & 
    \begin{tikzpicture}[baseline={([yshift=-.3ex]current bounding box.center)}]
    \draw[black](-1,1) -- (1,-1);
    \draw[black] (-1, -1) -- (1, 1);
    \draw[red, line width = 2.3] (-1,1) -- (1, -1);
    \end{tikzpicture}
    & 
    \begin{tikzpicture}[baseline={([yshift=-.3ex]current bounding box.center)}]
    \draw[black](-1,1) -- (1,-1);
    \draw[black] (-1, -1) -- (1, 1);
    \draw[red, line width = 2.3] (-1,-1) -- (1, 1);
    \end{tikzpicture}& 
    \begin{tikzpicture}[baseline={([yshift=-.3ex]current bounding box.center)}]
    \draw[black](-1,1) -- (1,-1);
    \draw[black] (-1, -1) -- (1, 1);
    \draw[red, line width = 2.3] (-1,1) -- (0, 0) -- (1, 1);
    \end{tikzpicture}& 
    \begin{tikzpicture}[baseline={([yshift=-.3ex]current bounding box.center)}]
    \draw[black](-1,1) -- (1,-1);
    \draw[black] (-1, -1) -- (1, 1);
    \draw[red, line width = 2.3] (-1,-1) -- (0, 0) -- (1, -1);
    \end{tikzpicture}\\
     \hline
     \;$a_1^{\hh}$\; & $a_2^{\hh}$  & $b_1^{\hh}$ & $b_2^{\hh}$ & $c_1^{\hh}$ & $c_2^{\hh}$ \\
    \hline
    \end{tabularx}
    \caption{Horizontal cross vertices}
    \label{fig:hcross}
\end{figure}
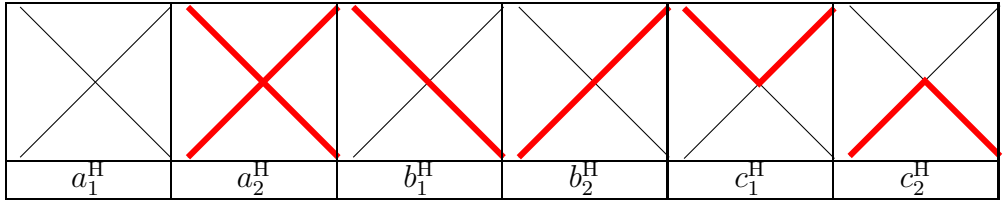

\begin{figure}
    \centering
    \begin{tabularx}{0.8\textwidth} { 
      | >{\centering\arraybackslash}X 
      | >{\centering\arraybackslash}X 
       | >{\centering\arraybackslash}X 
         | >{\centering\arraybackslash}X 
           | >{\centering\arraybackslash}X 
      | >{\centering\arraybackslash}X | }
     \hline
     
    \begin{tikzpicture}[baseline={([yshift=-.3ex]current bounding box.center)}]
    \draw[black](-1,1) -- (1,-1);
    \draw[black] (-1, -1) -- (1, 1);
    \end{tikzpicture} &
    \begin{tikzpicture}[baseline={([yshift=-.3ex]current bounding box.center)}]
    \draw[black](-1,1) -- (1,-1);
    \draw[black] (-1, -1) -- (1, 1);
    \draw[red, line width = 2.3](-1,1) -- (1,-1);
    \draw[red, line width = 2.3] (-1, -1) -- (1, 1);
    \end{tikzpicture} 
    & 
    \begin{tikzpicture}[baseline={([yshift=-.3ex]current bounding box.center)}]
    \draw[black](-1,1) -- (1,-1);
    \draw[black] (-1, -1) -- (1, 1);
    \draw[red, line width = 2.3] (-1,-1) -- (1, 1);
    \end{tikzpicture}
    & 
    \begin{tikzpicture}[baseline={([yshift=-.3ex]current bounding box.center)}]
    \draw[black](-1,1) -- (1,-1);
    \draw[black] (-1, -1) -- (1, 1);
    \draw[red, line width = 2.3] (-1,1) -- (1, -1);
    \end{tikzpicture}& 
    \begin{tikzpicture}[baseline={([yshift=-.3ex]current bounding box.center)}]
    \draw[black](-1,1) -- (1,-1);
    \draw[black] (-1, -1) -- (1, 1);
    \draw[red, line width = 2.3] (-1,1) -- (0, 0) -- (-1, -1);
    \end{tikzpicture}& 
    \begin{tikzpicture}[baseline={([yshift=-.3ex]current bounding box.center)}]
    \draw[black](-1,1) -- (1,-1);
    \draw[black] (-1, -1) -- (1, 1);
    \draw[red, line width = 2.3] (1,1) -- (0, 0) -- (1, -1);
    \end{tikzpicture}\\
     \hline
     \;$a_1^{\vv}$\; & $a_2^{\vv}$  & $b_1^{\vv}$ & $b_2^{\vv}$ & $c_1^{\vv}$ & $c_2^{\vv}$ \\
    \hline
    \end{tabularx}
    \caption{Vertical cross vertices}
    \label{fig:vcross}
\end{figure}
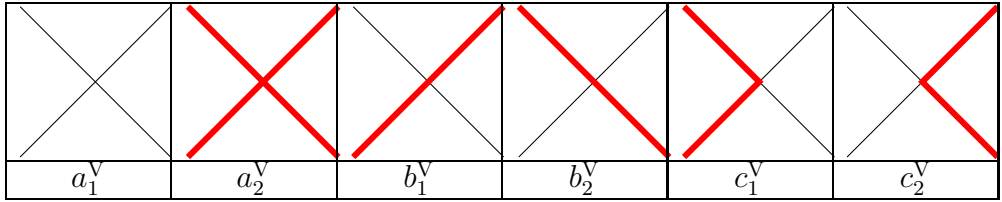

We now provide an extension of the Yang-Baxter equation, called the train argument. This argument is the basis of our construction of the switch operators.

\begin{lemma}\label{train arg}[Train argument] 
    Consider a six-vertex model and attach a cross vertex. Then the following relationship between partition functions holds. 

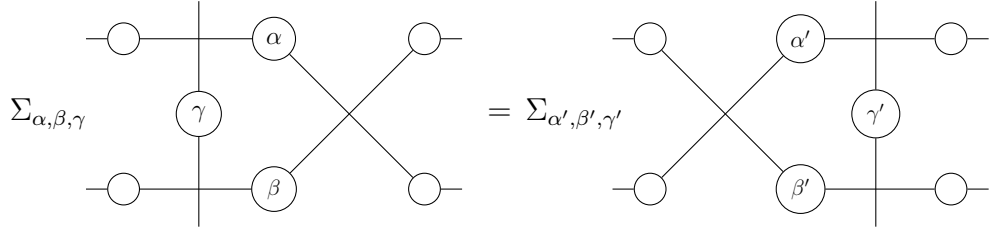
\begin{center}
\begin{tikzpicture}
\node[circle, minimum size=20pt] (sig2) at (7, 1) {$\displaystyle{\Sigma_{\alpha', \beta', \gamma'}}$};
\node[circle, minimum size=12pt, draw=black] (10) at (8,0) {};
\node[circle, minimum size=12pt, scale=0.7, draw=black] (20) at (10,0) {$\beta'$};
\node[circle, minimum size=12pt, scale=0.73, draw=black] (30) at (10,2) {$\alpha'$};
\node[circle, minimum size=12pt, draw=black] (40) at (8,2) {};
\node[circle, minimum size=12pt] (50) at (4, 3) {};
\node[circle, minimum size=12pt, scale=0.7, draw=black] (60) at (11,1) {$\gamma'$};
\node[circle, minimum size=12pt, draw=black] (70) at (12,2) {};
\node[circle, minimum size=12pt] (80) at (11,-1) {};
\node[circle, minimum size=12pt, draw=black] (90) at (12,0) {};
\draw (7.5,2) -- (40) -- (20) -- (90);
\draw (7.5,0) -- (10) -- (30) -- (70);
\draw (11,-0.5) -- (60) -- (11, 2.5);
\draw (70) -- (12.5, 2);
\draw (90) -- (12.5, 0);

\node[circle, minimum size=20pt] (eq) at (6, 1) {$=$};

\node[circle, minimum size=20pt] (sig1) at (0, 1) {$\displaystyle{\Sigma_{\alpha, \beta, \gamma}}$};
\node[circle, minimum size=12pt, draw=black] (1) at (1,0) {};
\node[circle, minimum size=12pt, scale=0.7, draw=black] (2) at (3,0) {$\beta$};
\node[circle, minimum size=12pt, scale=0.8, draw=black] (3) at (3,2) {$\alpha$};
\node[circle, minimum size=12pt, draw=black] (4) at (1,2) {};
\node[circle, minimum size=12pt] (5) at (4, 3) {};
\node[circle, minimum size=12pt, scale=0.8, draw=black] (6) at (2,1) {$\gamma$};
\node[circle, minimum size=12pt, draw=black] (7) at (5,2) {};
\node[circle, minimum size=12pt] (8) at (4,-1) {};
\node[circle, minimum size=12pt, draw=black] (9) at (5,0) {};
\draw (0.5,2) -- (4) -- (3) -- (9);
\draw (0.5,0) -- (1) -- (2) -- (7);
\draw (2,-0.5) -- (6) -- (2,2.5);
\draw (7) -- (5.5, 2);
\draw (9) -- (5.5, 0);
\end{tikzpicture}
\end{center}
    
    For example, 
    \[
        a_1^{\hh}(1,2)Z(1,2; k; \emptyset, \beta^t, (1,2), \beta^b) = a_2^{\hh}(1,2)Z(2,1; k; \emptyset; \beta^t, (1,2); \beta^b).
    \]
    This argument also holds in the vertical case by instead using the vertical Yang-Baxter equation. Below is an example of the mechanics of the train argument.
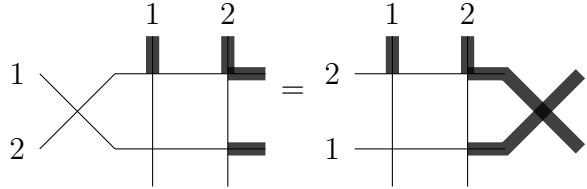
\[
    \vcenter{\hbox{\begin{tikzpicture}[scale=0.5]
        
        \foreach \i in {1,2}{
        \pgfmathtruncatemacro\row{int(2-\i+1)}
        \draw (0,2*\i-1) node[xshift=-1.3cm] {$\row$} -- (4, 2*\i-1);
    }

    \foreach \i in {1,2}{
        \pgfmathtruncatemacro\col{\i)}
        \draw (2*\i-1,0) -- (2*\i-1,4) node[yshift=0.3cm] {$\col$};
    }
        \draw[black, thin] (-2,1) to (0,3);
        \draw[black, thin] (-2,3) to (0,1);
        
        \draw[path] (3,1) to (4,1);
        \draw[path] (3,3) to (4,3);
        
        \draw[path] (1,3) to (1,4);
        \draw[path] (3,3) to (3,4);

    \end{tikzpicture}}}
=
\vcenter{\hbox{\begin{tikzpicture}[scale=0.5]
        \foreach \i in {1,2}{
        \pgfmathtruncatemacro\row{int(\i)}
        \draw (0,2*\i-1) node[xshift=-0.3cm] {$\row$} -- (4, 2*\i-1);
    }

    \foreach \i in {1,2}{
        \pgfmathtruncatemacro\col{\i)}
        \draw (2*\i-1,0) -- (2*\i-1,4) node[yshift=0.3cm] {$\col$};
    }
        \draw[path] (3,3) to (4,3) to (6,1);
        \draw[path] (3,1) to (4,1) to (6,3);
        
        \draw[path] (1,3) to (1,4);
        \draw[path] (3,3) to (3,4);

    \end{tikzpicture}}}
\]
\end{lemma}
\begin{proof}
    Since we assume that the weights of the horizontal (vertical) cross vertices do not denend on the column (row), we can repeately apply the Yang-Baxter equation until we reach the opposite boundary. 
\end{proof}

\section{Switch operators}
In this section we define the generalized Demazure operators, which we call the switch operators, that can be applied to both horizontal and vertical boundaries. We then axiomatically develop the algebra of partition functions under the switch operators. 

Let
\begin{align*}
    a_1, a_2, b_1, b_2, c_1, c_2 &\colon \mathbb{Z} \times \mathbb{Z} \to \mathbb{C},\\
    a_1^{\hh},a_2^{\hh},b_1^{\hh},b_2^{\hh},c_1^{\hh},c_2^{\hh} &\colon \mathbb{Z} \times \mathbb{Z} \to \mathbb{C},\\
    a_1^{\vv},a_2^{\vv},b_1^{\vv},b_2^{\vv},c_1^{\vv},c_2^{\vv} &\colon \mathbb{Z} \times \mathbb{Z} \to \mathbb{C},
\end{align*}
be the integrable weight functions, where $b_1^{\hh},b_2^{\hh},b_1^{\vv},b_2^{\vv}$ are non-zero, and where
\begin{align*}
    a_1^{\hh}(i+1,i)a_1^{\hh}(i,i+1) = b_1^{\hh}(i+1,i)b_2^{\hh}(i+1,i) + c_1^{\hh}(i+1,i)c_2^{\hh}(i+1,i) \\
    a_1^{\vv}(i+1,i)a_1^{\vv}(i,i+1) = b_1^{\vv}(i+1,i)b_2^{\vv}(i+1,i) + c_1^{\vv}(i+1,i)c_2^{\vv}(i+1,i). 
\end{align*}

Note that the last condition is equivalent to assuming that the $R$-matrix of the cross-vertex is invertible. This is demonstrated in detail in \cite{N22}.

Let $\alpha = (\alpha_1,\dots,\alpha_d) \in \mathbb{N}^d$ be a strictly decreasing signature of length $d$. Let $\delta_d = (d,d-1,\dots,1)$ and let $n,m \geq d$. 

We consider the six-vertex model $\mathfrak{S}^{n,m}(\emptyset, \delta_d,\alpha,\emptyset)$, that is, a model with $n$ rows and $m$ columns. The lattice consists of an empty left boundary, a dense top boundary, a boundary of $\alpha$ on the right, and an empty bottom boundary.

Let $Z_\alpha$ denote the partition function of this model. The main aim of this section is to give connection between the partition functions $Z_\alpha$ for different $\alpha$'s. 

The partition functions $Z_\alpha$ depend on the spectral parameters $I = (i_1,\dots,i_n)$ and $J = (j_1,\dots,j_m)$. Let the permutation group $S_{\infty}$ act on the spectral parameters as follows: 
\[
    \pi I = (i_{\pi(1)},\dots,i_{\pi(n)}).
\]

We also define the action of the simple transposition $s_i$ on the rightmost boundary as follows:
\begin{enumerate}
    \item if $i,i+1 \in \alpha$, then $s_i \alpha = \alpha$;
    \item if $i,i+1 \not\in \alpha$, then $s_i \alpha = \alpha$;
    \item if $i \in \alpha$ but $i+1 \not\in \alpha$, then $s_i \alpha = \alpha'$, where $\alpha'$ has the value $i+1$ replaced by $i$;
    \item if $i+1 \in \alpha$ but $i \not\in \alpha$, then $s_i \alpha = \alpha'$, where $\alpha'$ has the value $i$ replaced by $i+1$.
\end{enumerate}

In each case, we use the train argument to derive information about the relations that must hold between different horizontal cross vertex weights. 

In the first case, we have
\[
    a_1^{\hh}(i+1,i)Z_\alpha(s_i\,x; y) = a_2^{\hh}(i+1,i)Z_\alpha(x;y).
\]
In the second case, 
\[
    Z_{\alpha}(s_i\,x; y) = Z_{\alpha}(x; y).
\]
That is, swapping empty rows has no effect on the partition function. The final two cases provide the most significant insights. In the third case, we have
\begin{equation}\label{third eq}
    a_1^{\hh}(i+1,i)Z_{\alpha}(s_i\,x;y) = b_2^{\hh}(i+1,i)Z_{s_i \alpha}(x;y) + c_1^{\hh}(i+1,i)Z_{\alpha}(x;y).
\end{equation}
Lastly, the fourth case gives us the following:
\begin{equation}\label{fourth eq}
    a_1^{\hh}(i+1, i)Z_{\alpha}(s_i\,x; y) = b_1^{\hh}(i+1, i)Z_{s_i\alpha}(x; y) + c_2^{\hh}(i+1, i)Z_{\alpha}(x; y).
\end{equation}

If we look at the third case in particular, we notice that we can rewrite the terms and define a Demazure-like operator:
\[
    \partial_i^{\hh} = \frac{a_1^{\hh}(i+1,i)s_i^{\hh} - c_1^{\hh}(i+1,i)}{b_2^{\hh}(i+1,i)},
\]
such that 
\[
    Z_{s_i\alpha}(x;y) = \partial_i^{\hh}(Z_\alpha(x;y)) = \left(\frac{a_1^{\hh}(i+1,i)s_i^{\hh} - c_1^{\hh}(i+1,i)}{b_2^{\hh}(i+1,i)}\right)Z_\alpha(x;y).
\]

If $\partial_i^{\hh}$ is defined as above, then the inverse $\overline{\partial}_i^{\hh} = (\partial_i^{\hh})^{-1}$ is defined by the fourth case: 
\[
    \overline{\partial}_i^{\hh} = \frac{a_1^{\hh}(i+1,i)s_i^{\hh} - c_2^{\hh}(i+1,i)}{b_1^{\hh}(i+1,i)}.
\]

By \cite{N22}, we have the following relations between the weights:

\begin{lemma}[\cite{N22}]\label{lem:weightsproperties}
    For all $i,j \in \mathbb{Z}$, we have
    \begin{align*}
        b_1^{\hh}(i,j) &= -b_1^{\hh}(j,i),\\
        b_2^{\hh}(i,j) &= -b_2^{\hh}(j,i),\\
        c_1^{\hh}(i,j) &= c_2^{\hh}(j,i),\\
        c_2^{\hh}(i,j) &= c_1^{\hh}(j,i).
    \end{align*}
    Moreover, we have
    \begin{equation}
        a_1^{\hh}(i,j)a_1^{\hh}(j,i) + b_1^{\hh}(i,j)b_2^{\hh}(i,j) = c_1^{\hh}(i,j)c_2^{\hh}(i,j).
    \end{equation}
\end{lemma}

The property that these two operators are inverses of each other follows from the relations on the weights. 
By definition,
    \begin{align*}
        \partial_i^{\hh}\overline{\partial}_i^{\hh} &= \frac{a_1(i+1,i)s_i^{\hh}-c_1^{\hh}(i+1,i)}{b_2^{\hh}(i+1,i)}\frac{a_1^{\hh}(i+1,i)s_i^{\hh} - c_2^{\hh}(i+1,i)}{b_1^{\hh}(i+1,i)} \\
        &= \frac{c_1^{\hh}(i+1,i)c_2^{\hh}(i+1,i)-a_1^{\hh}(i+1,i)a_1^{\hh}(i,i+1)}{b_1^{\hh}(i+1,i)b_2^{\hh}(i+1,i)}\\
        &= 1.
    \end{align*}
Throughout the computation, we use the properties from \Cref{lem:weightsproperties}.

We use a similar approach to define the vertical switch operator. Let $s_j^{\vv}$ be the vertical analog of $s_i^{\hh}$, where $s_j^{\vv}$ acts by transposing the spectral parameters corresponding to columns $j$ and $j+1$. We will use the notation $s_j$ when it is clear we are acting in the vertical case.

We may now define the operators we have derived from our applications of the train argument.
\begin{definition}\label{5.1}
    The \textit{switch operators} $\partial_i^{\hh},\partial_i^{\vv}$ and their inverses $\overline{\partial}_i^{\hh}$, $\overline{\partial}_j^{\vv}$ are defined as
    \begin{align*}
        \partial_i^{\hh} = \dfrac{a_1^{\hh}(i+1,i)s_i^{\hh} - c_1^{\hh}(i+1, i)}{b_2^{\hh}(i+1,i)}, \quad 
        \overline{\partial}_i^{\hh} = \frac{a_1^{\hh}(i+1,i)s_i^{\hh} - c_2^{\hh}(i+1,i)}{b_1^{\hh}(i+1,i)},\\
        \partial_j^{\vv} = \dfrac{a_1^{\vv}(j+1, j)s_j^{\vv} - c_1^{\vv}(j+1, j)}{b_1^{\vv}(j+1, j)}, \quad \overline{\partial}_j^{\vv} = \dfrac{a_1^{\vv}(j+1, j)s_j^{\vv} - c_2^{\vv}(j+1, j)}{b_2^{\vv}(j+1, j)}.
    \end{align*}

    Consider the general partition function $Z_{\alpha,\beta}(I; J)$. The switch operators act on $Z_{\alpha,\beta}(I; J)$ by switching the boundaries of adjacent rows and columns:
    \begin{align*}
        Z_{s_i\alpha,\beta}(I;J) &= 
        \begin{cases}
            \partial_i^{\hh}(Z_\alpha(I;J)) & \text{if $i \in \alpha, i+1 \notin \alpha$} \\[12pt]
            \overline{\partial}_i^{\hh}(Z_\alpha(I;J)) & \text{if $i \notin \alpha, i + 1 \in \alpha$}
        \end{cases},\\
        Z_{\alpha,s_j\beta}(x; y) &= 
        \begin{cases}
            \partial_j^{\vv}(Z_{\beta}(I; J)) & \text{if $j \in \beta, j+1 \notin \beta$} \\[12pt]
            \overline{\partial}_j^{\vv}(Z_{\beta}(I; J)) & \text{if $j \notin \beta, j+1 \in \beta$}
        \end{cases}.
    \end{align*}
\end{definition}

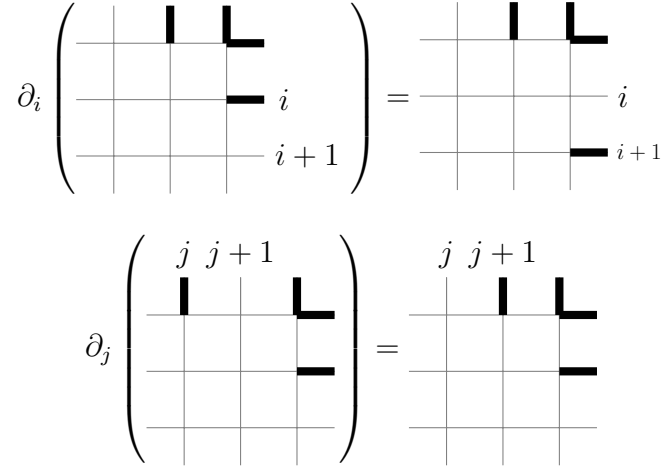
\begin{figure}[H]
    \centering
    $$ 
    \partial_i \left(
    \begin{tikzpicture}[baseline={([yshift=-.6ex]current bounding box.center)}, scale = 0.5]
    \draw[step=1.5cm,gray] (0.5,0.5) grid (5.5,5.5);
    \draw[black, line width=3.1] (4.5, 5.5) -- (4.5, 4.5);
    \draw[black, line width=3.1] (3, 5.5) -- (3, 4.5);
    \draw[black, line width=3.1] (4.5, 4.5) -- (5.5, 4.5);
    \draw[black,line width=3.1] (4.5, 3) -- (5.5, 3)node[right, black] {$i$};
    \draw[gray, opacity=0] (4.5, 1.5) -- (5.5, 1.5)node[right, black, opacity=1] {$i+1$};
    \end{tikzpicture} \right) \Large = 
    \begin{tikzpicture}[baseline={([yshift=-.6ex]current bounding box.center)}, scale = 0.5]
    \draw[step=1.5cm,gray] (0.5,0.5) grid (5.5,5.5);
    \draw[black, line width=3.1] (4.5, 5.5) -- (4.5, 4.5);
    \draw[black, line width=3.1] (3, 5.5) -- (3, 4.5);
    \draw[black, line width=3.1] (4.5, 4.5) -- (5.5, 4.5);
    \draw[gray, opacity=0] (4.5, 3) -- (5.5, 3)node[right, black, opacity=1, scale=0.7] {$i$};
    \draw[black, line width=3.1] (4.5, 1.5) -- (5.5, 1.5)node[right, black, scale=0.5] {$i+1$};
    \end{tikzpicture}
    $$
    
    $$ 
    \partial_j \left(
    \begin{tikzpicture}[baseline={([yshift=-.6ex]current bounding box.center)}, scale = 0.5]
    \draw[step=1.5cm,gray] (0.5,0.5) grid (5.5,5.5);
    \draw[black, line width=3.1] (1.5, 5.5) node[yshift=0.3cm] {$j$} -- (1.5, 4.5);
    \draw[black, line width=3.1] (4.5, 5.5) -- (4.5, 4.5);
    \draw[black, line width=3.1] (4.5, 4.5) -- (5.5, 4.5);
    \draw[black,line width=3.1] (4.5, 3) -- (5.5, 3);
    \draw[gray, opacity=0] (4.5, 1.5) -- (5.5, 1.5);
    \draw[gray, opacity=1] (3, 5.5) node[yshift=0.3cm, black] {$j+1$} -- (3, 4.5);
    \end{tikzpicture} \right) = 
    \begin{tikzpicture}[baseline={([yshift=-.6ex]current bounding box.center)}, scale = 0.5]
    \draw[step=1.5cm,gray] (0.5,0.5) grid (5.5,5.5);
    \draw[gray, opacity=1] (1.5, 5.5) node[yshift=0.3cm, black] {$j$} -- (1.5, 4.5);
    \draw[black, line width=3.1] (4.5, 5.5) -- (4.5, 4.5);
    \draw[black, line width=3.1] (3, 5.5) node[yshift=0.3cm, black] {$j+1$} -- (3, 4.5);
    \draw[black, line width=3.1] (4.5, 4.5) -- (5.5, 4.5);
    \draw[gray, opacity=0] (4.5, 3) -- (5.5, 3);
    \draw[black, line width=3.1] (4.5, 3) -- (5.5, 3);
    \end{tikzpicture}
    $$
    \caption{The effect of the horizontal and vertical switch operators on a $3 \times 3$ lattice model.}
\end{figure}

Note that we can use a composition of simple reflections $s_i$ to bring any signature $(n,n-1,\dots,1)$ to $\alpha$. For example, if we look at the horizontal direction:
\[
    (3,2,1) \xrightarrow{s_3} (4,2,1) \xrightarrow{s_4} (5,2,1) \xrightarrow{s_2} (5,3,1) \xrightarrow{s_1} (5,3,2).
\]

Hence, using the operator $\partial_i^{\hh}$, we can express $Z_\alpha$ in terms of $Z_\delta$, where $\delta = (n,n-1,\dots,1)$. We call $Z_\delta$ the \textit{base case}. Note that $Z_\delta = Z_n$ from \Cref{def:dwbc}. 

Thus, we can express $Z_{\alpha}$ as
\[
    Z_{\alpha} = \partial_{i_1}^{\hh} \partial_{i_2}^{\hh} \dots \partial_{i_k}^{\hh} (Z_\delta),
\] 
for some $\partial_{i_m}^{\hh}$'s. The vertical case follows this same property, but instead prescribes a strictly decreasing signature of $\beta = (\beta_1, \dots, \beta_n) \in \mathbb{N}^n$ to the top boundary.

We express this concept with more precise notation. Let $D = (d, d-1, \cdots, 1)$ where $d$ is the length of the signature $\alpha$.
Then let
\begin{equation*}
    \partial_{\alpha}^{\hh} = \prod_{k=1}^d \prod_{\ell = 1}^{ \alpha_{d-k+1} - D_k} \partial_{\alpha_{d-k+1}-\ell}^{\hh}.
\end{equation*}
Lastly, let $M$ denote the number of columns in the lattice, and let
\begin{equation*}
    \partial_{\beta}^V = \prod_{k=1}^d \prod_{\ell = \beta_k}^{M-k}\partial_{\ell}^{\vv}. 
\end{equation*}

Then, consider the partition function being normalized in the $a_1$ weight, so that the addition of extra empty columns or rows does not impact the partition function. 

Before we reach our theorem, we develop some useful notation regarding repeated use of the switch operators. For $i<j$, denote by $\partial_{[i,j]}$ the operator sequence $\partial_{j-1,j}\partial_{j-2,j-1}\dots \partial_{i,i+1}$ and $s_{[i,j]}$ the sequence $s_{j-1,j}s_{j-2,j-1}\dots s_{i,i+1}$. Let $\partial_{\alpha}^{\hh}$ be the expression $\partial_{[1,\alpha_n]}^{\hh}\partial_{[2,\alpha_{n-1}]}^{\hh}\dots \partial_{[n,\alpha_1]}^{\hh}$. Similarly, we denote $\partial_{\alpha}^{\vv}$ to be $\partial_{[1,\alpha_n]}^{\vv}\partial_{[2,\alpha_{n-1}]}^{\vv}\dots \partial_{[n,\alpha_1]}^{\vv}$.

\begin{theorem}\label{thm:reductiontobasecase}
    The partition function $Z_{\alpha,\beta}$ with right boundary $\alpha$ and top boundary $\beta$ is expressed in terms of the base case $Z_{\delta,\delta}$ as follows:
    \[
        Z_{\alpha,\beta} = \partial_\alpha^{\hh} \partial_\beta^{\vv}\left(Z_{\delta,\delta}\right).
    \]
\end{theorem}
\begin{proof}
    We have
    \begin{align*}
        \partial_{\alpha}^{\hh}(Z_{\delta,\delta}) &= \partial_{[1,\alpha_n]}\partial_{[2,\alpha_{n-1}]}\dots \partial_{[n,\alpha_1]}(Z_{\delta,\delta}) \\
        &= \partial_{[1,\alpha_n]}\partial_{[2,\alpha_{n-1}]}\dots \partial_{[n-1,\alpha_2]}(Z_{s_{[n,\alpha_1]}\delta,\delta}) \\
        &= \dots \\
        &= Z_{s_{[1,\alpha_n]}s_{[2,\alpha_{n-1}]}\dots s_{[n,\alpha_1]}\delta, \beta}. 
    \end{align*}
    Since $s_{[1,\alpha_n]}s_{[2,\alpha_{n-1}]}\dots s_{[n,\alpha_1]}\delta = \alpha$, we get the result. The vertical case is analogous. 
\end{proof}

Thus we have proved it is possible to express the partition function of a lattice with arbitrary boundaries in terms of the operators acting on the base case.

\begin{example}
    Let $n = 5$. Take $Z_{\delta, \delta}$ in the $5 \times 5$ case, and take $Z_{\alpha, \beta}$ to be the partition function with signatures $\alpha = (5, 3, 2)$ and $\beta = (4, 2, 1)$:
    
    
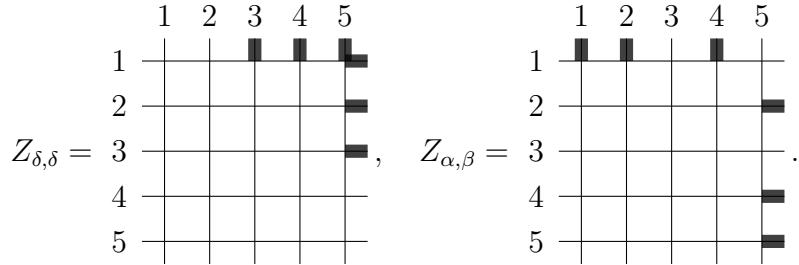
\begin{figure}[H]
        \centering
        \[
        Z_{\delta, \delta} = 
        \begin{tikzpicture}[baseline={([yshift=-1em]current bounding box.center)}, scale = 0.3]
            \grid{5}{5}
            
            \draw[path] (5,9) to (5,10);
            \draw[path] (7,9) to (7,10);
            \draw[path] (9,9) to (9,10);
            
            \draw[path] (9,9) to (10,9);
            \draw[path] (9,7) to (10,7);
            \draw[path] (9,5) to (10,5);
        \end{tikzpicture}, \quad Z_{\alpha, \beta} =  \begin{tikzpicture}[baseline={([yshift=-1em]current bounding box.center)}, scale = 0.3]
            \grid{5}{5}
            
            \draw[path] (1,9) to (1,10);
            \draw[path] (3,9) to (3,10);
            \draw[path] (7,9) to (7,10);
            
            \draw[path] (9,7) to (10,7);
            \draw[path] (9,3) to (10,3);
            \draw[path] (9,1) to (10,1);
        \end{tikzpicture}.
        \]
    \end{figure}
    
    We can apply our switch operators first in the \textcolor{blue}{vertical case}:
    
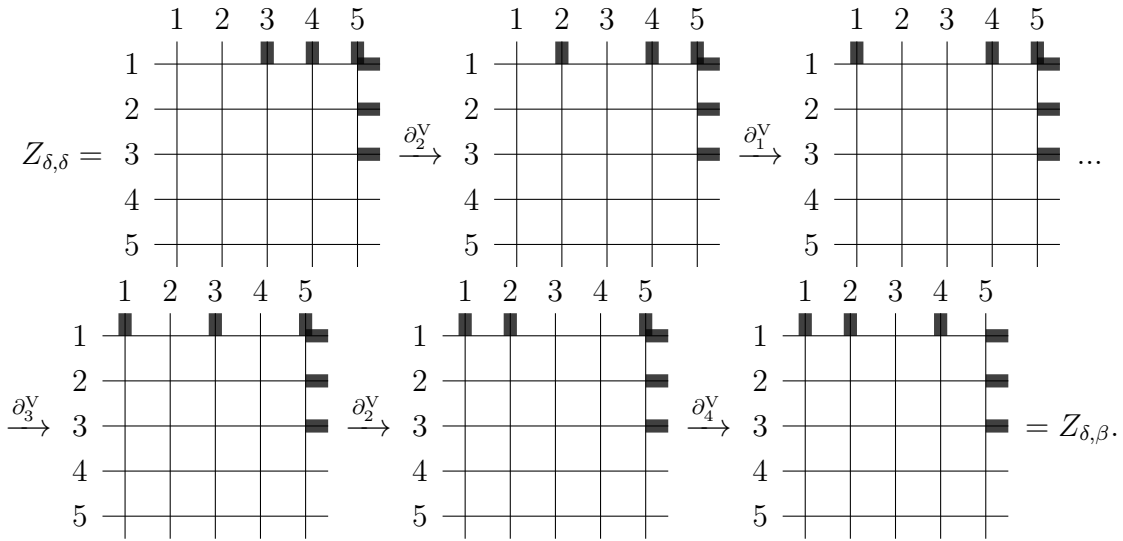
\begin{figure}[H]
        \centering
        $Z_{\delta, \delta} = $
        \begin{tikzpicture}[baseline={([yshift=-1em]current bounding box.center)}, scale = 0.3]
            \grid{5}{5}
            
            \draw[path] (5,9) to (5,10);
            \draw[path] (7,9) to (7,10);
            \draw[path] (9,9) to (9,10);
            
            \draw[path] (9,9) to (10,9);
            \draw[path] (9,7) to (10,7);
            \draw[path] (9,5) to (10,5);
        \end{tikzpicture} $\xrightarrow{\partial_2^{\vv}}$
        \begin{tikzpicture}[baseline={([yshift=-1em]current bounding box.center)}, scale = 0.3]
            \grid{5}{5}
            
            \draw[path] (3,9) to (3,10);
            \draw[path] (7,9) to (7,10);
            \draw[path] (9,9) to (9,10);
            
            \draw[path] (9,9) to (10,9);
            \draw[path] (9,7) to (10,7);
            \draw[path] (9,5) to (10,5);
        \end{tikzpicture} $\xrightarrow{\partial_1^{\vv}}$
        \begin{tikzpicture}[baseline={([yshift=-1em]current bounding box.center)}, scale = 0.3]
            \grid{5}{5}
            
            \draw[path] (1,9) to (1,10);
            \draw[path] (7,9) to (7,10);
            \draw[path] (9,9) to (9,10);
            
            \draw[path] (9,9) to (10,9);
            \draw[path] (9,7) to (10,7);
            \draw[path] (9,5) to (10,5);
        \end{tikzpicture} ...
        
        $\xrightarrow{\partial_3^{\vv}}$ \begin{tikzpicture}[baseline={([yshift=-1em]current bounding box.center)}, scale = 0.3]
            \grid{5}{5}
            
            \draw[path] (1,9) to (1,10);
            \draw[path] (5,9) to (5,10);
            \draw[path] (9,9) to (9,10);
            
            \draw[path] (9,9) to (10,9);
            \draw[path] (9,7) to (10,7);
            \draw[path] (9,5) to (10,5);
        \end{tikzpicture} $\xrightarrow{\partial_2^{\vv}}$
        \begin{tikzpicture}[baseline={([yshift=-1em]current bounding box.center)}, scale = 0.3]
            \grid{5}{5}
            
            \draw[path] (1,9) to (1,10);
            \draw[path] (3,9) to (3,10);
            \draw[path] (9,9) to (9,10);
            
            \draw[path] (9,9) to (10,9);
            \draw[path] (9,7) to (10,7);
            \draw[path] (9,5) to (10,5);
        \end{tikzpicture} $\xrightarrow{\partial_4^{\vv}}$
        \begin{tikzpicture}[baseline={([yshift=-1em]current bounding box.center)}, scale = 0.3]
            \grid{5}{5}
            
            \draw[path] (1,9) to (1,10);
            \draw[path] (3,9) to (3,10);
            \draw[path] (7,9) to (7,10);
            
            \draw[path] (9,9) to (10,9);
            \draw[path] (9,7) to (10,7);
            \draw[path] (9,5) to (10,5);
        \end{tikzpicture} $ = Z_{\delta, \beta}$.
    \end{figure}
    
    Then, we apply our switch operators in the \textcolor{red}{horizontal case}:
    
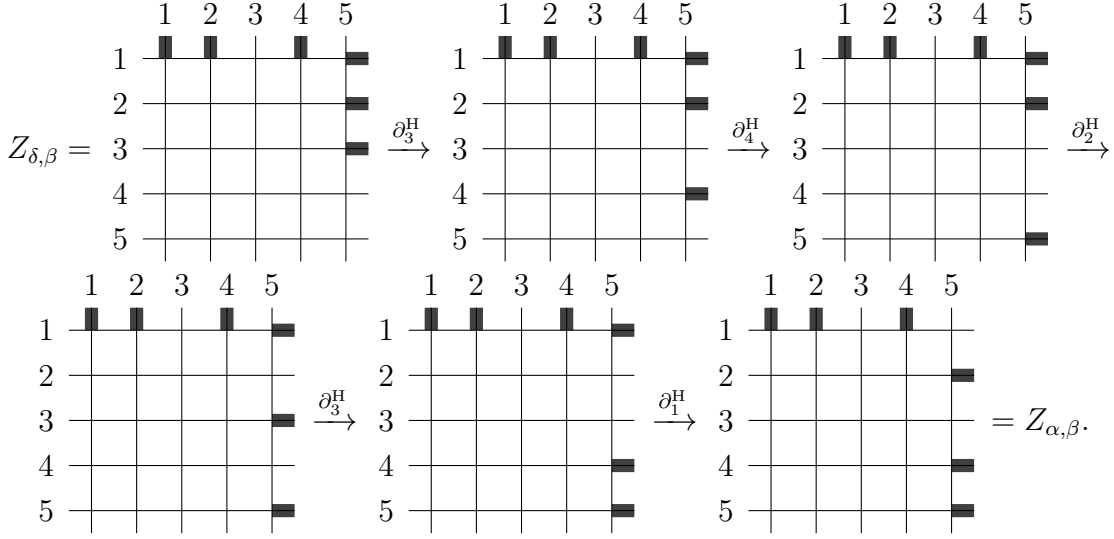
\begin{figure}[H]
        \centering
        $Z_{\delta, \beta} = $ \begin{tikzpicture}[baseline={([yshift=-1em]current bounding box.center)}, scale = 0.3]
            \grid{5}{5}
            
            \draw[path] (1,9) to (1,10);
            \draw[path] (3,9) to (3,10);
            \draw[path] (7,9) to (7,10);
            
            \draw[path] (9,9) to (10,9);
            \draw[path] (9,7) to (10,7);
            \draw[path] (9,5) to (10,5);
        \end{tikzpicture} $\xrightarrow{\partial_3^{\hh}}$
        \begin{tikzpicture}[baseline={([yshift=-1em]current bounding box.center)}, scale = 0.3]
            \grid{5}{5}
            
            \draw[path] (1,9) to (1,10);
            \draw[path] (3,9) to (3,10);
            \draw[path] (7,9) to (7,10);
            
            \draw[path] (9,9) to (10,9);
            \draw[path] (9,7) to (10,7);
            \draw[path] (9,3) to (10,3);
        \end{tikzpicture} $\xrightarrow{\partial_4^{\hh}}$
        \begin{tikzpicture}[baseline={([yshift=-1em]current bounding box.center)}, scale = 0.3]
            \grid{5}{5}
            
            \draw[path] (1,9) to (1,10);
            \draw[path] (3,9) to (3,10);
            \draw[path] (7,9) to (7,10);
            
            \draw[path] (9,9) to (10,9);
            \draw[path] (9,7) to (10,7);
            \draw[path] (9,1) to (10,1);
        \end{tikzpicture} $\xrightarrow{\partial_2^{\hh}}$
        
        \begin{tikzpicture}[baseline={([yshift=-1em]current bounding box.center)}, scale = 0.3]
            \grid{5}{5}
            
            \draw[path] (1,9) to (1,10);
            \draw[path] (3,9) to (3,10);
            \draw[path] (7,9) to (7,10);
            
            \draw[path] (9,9) to (10,9);
            \draw[path] (9,5) to (10,5);
            \draw[path] (9,1) to (10,1);
        \end{tikzpicture} $\xrightarrow{\partial_3^{\hh}}$
        \begin{tikzpicture}[baseline={([yshift=-1em]current bounding box.center)}, scale = 0.3]
            \grid{5}{5}
            
            \draw[path] (1,9) to (1,10);
            \draw[path] (3,9) to (3,10);
            \draw[path] (7,9) to (7,10);
            
            \draw[path] (9,9) to (10,9);
            \draw[path] (9,3) to (10,3);
            \draw[path] (9,1) to (10,1);
        \end{tikzpicture} $\xrightarrow{\partial_1^{\hh}}$
        \begin{tikzpicture}[baseline={([yshift=-1em]current bounding box.center)}, scale = 0.3]
            \grid{5}{5}
            
            \draw[path] (1,9) to (1,10);
            \draw[path] (3,9) to (3,10);
            \draw[path] (7,9) to (7,10);
            
            \draw[path] (9,7) to (10,7);
            \draw[path] (9,3) to (10,3);
            \draw[path] (9,1) to (10,1);
        \end{tikzpicture} $= Z_{\alpha, \beta}$.
    \end{figure}
    
    From the illustrations, we see that the sequence of applications of switch operators that takes us from the base case to $Z_{\alpha, \beta}$ is
    \begin{align*}
        Z_{\alpha, \beta} = \textcolor{red}{\partial_1^{\hh} \partial_3^{\hh} \partial_2^{\hh} \partial_4^{\hh} \partial_3^{\hh}} \textcolor{blue}{\partial_4^{\vv} \partial_2^{\vv} \partial_3^{\vv} \partial_1^{\vv} \partial_2^{\vv}}(Z_{\delta, \delta}),
    \end{align*}
    which exactly matches our expectations from Theorem 5.3 based off of $\partial_{\alpha}^{\hh}$ and $\partial_{\beta}^{\vv}$. 
\end{example}

We now present an application of these results. Let the weight functions be as follows:
\begin{equation}\label{eq:ffweights}
    \begin{aligned}
        a_1(i,j) &= 1-b_j x_i,\\
        a_2(i,j) &= y_i + a_j,
    \end{aligned} \quad
    \begin{aligned}
        b_1(i,j) &= 1+b_j y_i,\\
        b_2(i,j) &= x_i - a_j,
    \end{aligned} \quad
    \begin{aligned}
        c_1(i,j) &= 1-a_jb_j,\\
        c_2(i,j) &= x_i+y_i.
    \end{aligned}
\end{equation}

These weights are from \cite{N23}, where their partition functions were shown to generalize various families of the Schur functions. These functions are called \textit{free fermionic Schur functions}. By Corollary 2.14, the partition function with the domain wall boundary conditions is 
\[
    Z_n^{\operatorname{DWBC}}(x,y; a,b) = \prod_{i<j}(x_i-y_j)(1-a_ib_j).
\]
Let $\alpha = (\alpha_1,\alpha_2,\dots,\alpha_n)$. Consider the partition function $Z_{\alpha,\delta}(x,y; a,b)$ which generalizes various non-supersymmetric Schur functions. Then we have the following result:

\begin{proposition}
    We have the following evaluation of the partition function:
    \[
        Z_{\alpha,\delta}(x,y; a,b) = \partial_\alpha^{\vv}\left(\prod_{i=1}^{n}\prod_{j=n+1}^{\alpha_1}(1-b_jx_i)\prod_{i<j}(x_i-y_j)(1-a_ib_j)\right).
    \]
\end{proposition}
\begin{proof}
    By \Cref{thm:reductiontobasecase} and the explicit value of the partition function with the domain wall boundary conditions. Note that the extra factor comes from the ``empty'' sites in the six vertex model.
\end{proof}

This result generalizes the relation for the factorial Schur functions as explored in equations (18) and (19) in \cite{BMN14}. In particular, the equation (19) can be written in terms of the switch operators as 
\[
    s_\mu(z|\sigma_i \alpha) = \partial_i^{\vv}(s_\lambda(z|\alpha)).
\]

Similarly to Corollary 1 in \cite{BMN14}, we prove that the partition function is asymptotically symmetric in the column parameters.

\begin{corollary}
    The free fermionic Schur functions $Z_{\alpha,\delta}$ are asymptotically symmetric in variables $a_j,b_j$.
\end{corollary}
\begin{proof}
    Indeed, for large enough indices which exceed all partis of $\alpha$, the switch operators give the simplified expression
    \[
        Z_{\alpha,\beta}(x,y; a,b) = Z_{\alpha,\beta}(x,y; s_i a, s_i b).
    \]
    Hence, the partition function is asymptotically symmetric in the column parameters.
\end{proof}

We note that \Cref{thm:reductiontobasecase} provides the explicit expression for the partition functions $Z_{\alpha,\beta}(x,y; a,b)$. While the specializations $\beta = \delta$ and $\alpha = \delta$ produce the generalizations of the factorial Schur functions, the meaning of the function $Z_{\alpha,\beta}(x,y; a,b)$ remains unclear. This function could be seen as two sided interpolation of the two kinds of factorial Schur functions (or their generalizations).

\bibliographystyle{alphaurl}
\bibliography{switch.bib}

\end{document}